\documentclass[reqno,draft]{amsart}

\theoremstyle{plain}
\newtheorem{lemma}{Lemma}
\newtheorem{theorem}[lemma]{Theorem}

\newtheorem{proposition}[lemma]{Proposition}

\theoremstyle{remark}
\newtheorem{remark}{Remark}

\newcommand*  {\R} {{\mathbb R}}

\def\la{\lambda}
\def\Dd{\Delta}

\def\Om{\Omega}

\def\pp{\partial}

\def\fy{\varphi}
\newcommand{\EQ}[1]{\begin{equation}\begin{split} #1 \end{split}\end{equation}}
\def\pr{\\&}

\begin{document}

\vskip 0.125in

\title[Finite-time Blowup for the Inviscid Primitive Equations]
{Finite-time Blowup for the Inviscid Primitive Equations of Oceanic and Atmospheric Dynamics}

\date{October 26, 2012}
%\thanks{\textit{ }}

\author[C. Cao]{Chongsheng Cao}
\address[C. Cao]
{Department of Mathematics  \\
Florida International University  \\
University Park  \\
Miami, FL 33199, USA.} \email{caoc@fiu.edu}

\author[S. Ibrahim ]{Slim Ibrahim}
\address[S. Ibrahim]
{Department of Mathematics and Statistics \\
University of Victoria \\
PO BOX 3060, STN CSC\\
Victoria, BC  V8W 3R4,
Canada} \email{ibrahims@uvic.ca}

\author[K. Nakanishi]{Kenji Nakanishi}
\address[K. Nakanishi]
{Department of Mathematics  \\
Kyoto University \\
Kyoto 606-8502, JAPAN} \email{n-kenji@math.kyoto-u.ac.jp}

\author[E.S. Titi]{Edriss S. Titi}
\address[E.S. Titi]
{ Department of Computer Science and Applied Mathematics \\
Weizmann Institute of Science  \\
Rehovot 76100, Israel. {\bf Also:} Department of Mathematics \\
and  Department of Mechanical and  Aerospace Engineering \\
University of California \\
Irvine, CA  92697-3875, USA. Fellow of the Center of Smart Interfaces, Technische Universit\"at Darmstadt, Germany.} \email{etiti@math.uci.edu}
\email{edriss.titi@weizmann.ac.il}

\begin{abstract}

In an earlier work we have shown the global (for all initial data and all time)  well-posedness of strong solutions to the three-dimensional viscous primitive equations of large scale oceanic and atmospheric dynamics.  In this paper we show that for certain class of initial data the corresponding smooth solutions of the  inviscid  (non-viscous) primitive equations  blow up in finite time.  Specifically, we consider the three-dimensional  inviscid  primitive equations in a
three-dimensional infinite horizontal channel,
subject to periodic boundary conditions in the
horizontal directions, and with  no-normal flow  boundary conditions on the solid, top and bottom, boundaries. For certain class of initial data we reduce this system into the two-dimensional  system of primitive equations in an infinite horizontal strip with the same type of boundary conditions; and then show that for specific sub-class of initial data the corresponding  smooth solutions of the  reduced inviscid two-dimensional system develop singularities in finite time.
\end{abstract}

\maketitle



\section{Introduction}   \label{S-1}

The three-dimensional  primitive equations for large scale oceanic and atmospheric dynamics are given
by the system of partial differential equations:
\begin{eqnarray}
&&\hskip-.8in u_t + u u_x + vu_y + w u_z+ p_x - R v = \nu_H
\Dd_H  u + \nu_3 u_{zz},  \label{EQ-1}  \\
&&\hskip-.8in v_t + u v_x + v v_y +
w v_z+ p_y + R u = \nu_H
\Dd_H  v + \nu_3 v_{zz},   \label{EQ-2}  \\
&&\hskip-.8in  p_z   +
T =0,    \label{EQ-3}  \\
&&\hskip-.8in T_t  + uT_x+v T_y +
w T_z  = Q+ \kappa_H \Dd_H T + \kappa_3 T_{zz}, \label{EQ-4}  \\
&&\hskip-.8in u_x+ v_y + w_z
=0,   \label{EQ-5}
\end{eqnarray}
 which are supplemented with the initial value $(u_0,v_0,T_0)$, and satisfy  the relevant  geophysical boundary conditions (see, e.g., \cite{Majda,Pedlosky,Salmon}. Here $\Dd_H = \pp_{xx} + \pp_{yy}$ denotes the horizontal Laplacian operator.
The global well-posedness (for all time and for all initial data) of strong solutions to  the above three-dimensional system, subject to the relevant geophysical boundary conditions, has been proven first in \cite{Cao-Titi1}
with full viscosity  (i.e. $\nu_H>0$ and $\nu_3>0$) and full diffusion (i.e. $\kappa_H>0$ and $\kappa_3>0$), (see also, \cite{Kukavica-Ziane} for the case of Dirichlet boundary conditions).  This result has been improved recently in \cite{Cao-Titi2} for the case of full viscosity (i.e. $\nu_H>0$ and $\nu_3>0$) and only partial anisotropic
vertical diffusion (i.e. $\kappa_H=0$ and $\kappa_3>0$) which stands for the vertical  eddy heat diffusivity turbulence mixing coefficient (see, e.g., \cite{GA84}, \cite{GC93}, for the geophysical justification). In the above results, the
Coriolis forcing term, with rotation parameter $R$,  did not play any role in proving the global regularity. This is contrary to the cases  of the three-dimensional fast rotating Euler,  Navier--Stokes and Boussinesq equations by \cite{BMN97,BMN99,BMN99a,BMN00} where the authors take full advantage of the absence of resonances between the fast rotation and the nonlinear advection. This absence of resonances at the limit of fast rotation   leads to strong dispersion and averaging mechanism (at the limit of fast rotation for large values   $R$ depending on the size of the initial data) that weakens the nonlinear effects and hence allows for establishing the global regularity result in the viscous Navier-Stokes case, and prolongs the life-space of the solution in the Euler case (see also \cite{Chemin-book,Embid-Majda} and references therein; in addition,  see \cite{Babin-Ilyin-Titi} for simple examples demonstrating the above mechanism).

In geophysical situations it is observed that the viscosity coefficients are very small, and that in fact $\nu_3 \ll \nu_H \ll 1$. Motivated by this observation and the  above discussion it will become interesting to  know of whether the inviscid  (non-viscous)   primitive equations are globally regular or that they develop singularity (blow up) in finite time. Since the rotation term did not play any role in establishing the global regularity in the viscous cases (cf. \cite{Cao-Titi1,Cao-Titi2})  one might as well ask the above question without the rotation term (i.e., consider the case with $R=0$). Therefore,  we will consider in this paper the inviscid primitive equations without the
Coriolis rotation term, and we will   show that for certain class of smooth initial data their corresponding
smooth solutions will develop a singularity (blowup) in finite time.

In order to establish our result we will assume that we are given a smooth solution to the inviscid primitive equation. In section \ref{S-2} we will derive a reduced equation that this smooth solution will satisfy. In section \ref{S-3} we will follow \cite{Childress} (see also \cite{Okamoto}) to show that for certain class of initial data the corresponding solutions to this reduced equation blow up in finite time. In section \ref{S-4} we provide a family of initial data whose corresponding  smooth solutions to the inviscid primitive equations  blow up in finite time. It is worth mentioning that similar approach has been introduced in \cite{E-Engquist}  to show the blowup for the Prandtl equation of the boundary layer in the upper half-space.

\section{Derivation of a reduced equation}   \label{S-2}
In this section we will assume that we are given a conveniently smooth solution to the inviscid primitive equations. We will derive a reduced equation that this solution must satisfy by requiring that  the solution fulfills  certain symmetry conditions.  Eventually,  in section \ref{S-3}, we will  show that for certain class of initial data the corresponding solutions of this reduced equation develop singularity in finite time.

First,  let us consider the inviscid primitive equations without the Coriolis force:
\begin{eqnarray}
&&\hskip-.8in u_t + u u_x + vu_y + w u_z+ p_x = 0,
\label{EEQ-1}  \\
&&\hskip-.8in v_t + u v_x + v v_y + w v_z+
p_y  = 0,   \label{EEQ-2}  \\
&&\hskip-.8in p_z   +  T =0,    \label{EEQ-3}
\\ &&\hskip-.8in T_t + uT_x+v T_y + w T_z  =  \kappa_H \Dd_H T + \kappa_3 T_{zz}, \label{EEQ-4}  \\ &&\hskip-.8in u_x+ v_y + w_z =0,   \label{EEQ-5}
\end{eqnarray}
in the horizontal  channel $\Om= \{ (x,y,z): 0\leq z \leq H, \, (x,y)\in \R^2 \}$; subject
to the boundary conditions: no-normal flow and no heat flux in the vertical
direction at the physical solid (top $z=H$ and bottom $z=0$) boundaries,  and periodic boundary conditions, say with period $L$,  in horizontal directions. Observe
that when the initial temperature $T_0=0$, then it is easy to see that any smooth solution to
system (\ref{EEQ-1})--(\ref{EEQ-5}), subject the above boundary conditions, must satisfy $T(x,y,z,t)\equiv 0$. Consequently, for any smooth solution with $T_0=0$ the velocity field satisfies the following
system:
\begin{eqnarray}
&&\hskip-.8in u_t + u u_x + vu_y + w u_z+ p_x = 0,
\label{EEEQ-1}  \\
&&\hskip-.8in v_t + u v_x + v v_y + w
v_z+ p_y  = 0,   \label{EEEQ-2}  \\
&&\hskip-.8in p_z  =0,
\label{EEEQ-3}  \\
&&\hskip-.8in u_x+ v_y + w_z =0\,.   \label{EEEQ-5}
\end{eqnarray}
Since our goal is to establish    the blowup for certain class of smooth solutions and  initial data, we will further simplify matters and restrict ourselves to smooth  solution to the above system that are  independent of the $y-$variable and that the
$y-$component of initial velocity vector filed  $v_0=0$. Once again it is   clear that any smooth unique solution to the above system with these properties will satisfy $v(x,z,t) \equiv 0,$ and that $u,w$ are functions of the spatial variables $(x,z)$ and of the time $t$ only.
Consequently, we restrict ourselves further and only consider solutions that satisfy the following two-dimensional sub-system:
\begin{eqnarray}
&&\hskip-.8in u_t + u u_x  + w u_z+ p_x  = 0,  \label{TEQ-1} \\
&&\hskip-.8in p_z   =0,    \label{TEQ-2}   \\
&&\hskip-.8in u_x + w_z
=0,   \label{TEQ-3}
\end{eqnarray}
in the two-dimensional  channel
\[
M= \{ (x,z): 0\leq z \leq H\,, x\in \R\},
\]
subject to the no-normal flow on the vertical direction, and periodic boundary conditions in horizontal $x-$direction, namely,
\begin{eqnarray}
&&\hskip-.8in \left. w \right|_{z=H}=\left. w \right|_{z=0}=0,
\label{PP-1}
\end{eqnarray}
and
\begin{eqnarray}
&&\hskip-.8in u(x+L,z,t)=u(x,z,t);\;  p(x+L,z,t)=p(x,z,t); \;w(x+L,z,t)=w(x,z,t).
\label{PP-2}
\end{eqnarray}
Observe that the space of periodic functions with respect to $x$ with the
following symmetry
\begin{eqnarray}
&&\hskip-.8in u(x,z,t)=-u(-x,z,t);  p(x,z,t)=p(-x,z,t); w(x,z,t)=w(-x,z,t),
\label{PP-3}
\end{eqnarray}
is invariant under the solution operator of system
(\ref{TEQ-1})--(\ref{TEQ-3}) subject to the boundary condition (\ref{PP-1})
and (\ref{PP-2}). Therefore, from now on we will restrict our discussion to the solutions of system (\ref{TEQ-1})--(\ref{PP-2}) that satisfy  the
symmetry condition (\ref{PP-3}).

Let us denote by
\begin{eqnarray}
&&\hskip-.8in \overline{\phi} (x,t)= \frac{1}{H} \int_0^H \phi(x, z,t) \; dz.
\label{PP-4}
\end{eqnarray}
From equation (\ref{TEQ-3}) and the boundary condition (\ref{PP-1}), we
have
\begin{eqnarray}
&&\hskip-.8in 0= w(x,H,t)-w(x,0,t) = \int_0^H w_z (x,z,t) \; dz =  - \int_0^H
u_x (x,z,t) \; dz.  \label{PP-5}
\end{eqnarray}
Thus,
\begin{eqnarray}
&&\hskip-.8in \overline{u}_x (x,t)= 0.  \label{PP-6}
\end{eqnarray}
By differentiating  equation (\ref{TEQ-1}), with respect to $x$, we obtain
\begin{eqnarray}
&&\hskip-.8in \frac{\pp u_x}{\pp t} + (u u_x)_x  + w_x u_z + w u_{zx}+
p_{xx}  = 0.  \label{DTEQ-1}
\end{eqnarray}
Equation (\ref{TEQ-2}) implies  that $p(x,z,t)$ is only function of
$(x,t)$, i.e., $p(x,z,t)=p(x,t)$. By averaging (\ref{DTEQ-1}) with respect to the  $z$ variable over $[0,H]$ and
using  (\ref{PP-6}), we obtain
\begin{eqnarray}
&&\hskip-.8in
 \overline{(u u_x)_x}  + \overline{w_x u_z} + \overline{w u_{xz}}+ p_{xx}  = 0.
 \label{DTEQ-2}
\end{eqnarray}
Notice that   (\ref{PP-1}) implies
\begin{eqnarray}
&&\hskip-.8in w_x(x,H,t)=w_x(x,0,t) = 0.  \label{PP-7}
\end{eqnarray}
Thanks to  (\ref{TEQ-3}) and (\ref{PP-7}), integrating by parts gives
\begin{eqnarray*}
&&\hskip-.8in
 \overline{w_x u_z} + \overline{w u_{xz}} = -\overline{w_{xz} u} -
 \overline{w_z u_{x}}   \\
&&\hskip-.8in = \overline{u_{xx} u} + \overline{ (u_x)^2} = \overline{ (u
u_x)_x}.
\end{eqnarray*}
As a result, (\ref{DTEQ-2}) can be reduced to
\begin{eqnarray}
&&\hskip-.8in
 2\overline{(u u_x)_x}  + p_{xx}  = 0.  \label{DTEQ-3}
\end{eqnarray}
Thus,
\begin{eqnarray}
&&\hskip-.8in
 2\overline{u u_x}  + p_{x}  = C(t),  \label{DTEQ-4}
\end{eqnarray}
for some function $C(t).$ Since  $p_x$ and $\overline{u u_x}$ are
odd functions, with respect to the variable $x$ (see condition (\ref{PP-3})), then
\[
\overline{u u_x}(0,t)=  p_{x} (0,t)=0.
\]
Therefore,
\begin{eqnarray}
&&\hskip-.8in C(t)= 2\overline{u u_x} (0, t) + p_{x} (0,t) = 0,
\label{DTEQ-5}
\end{eqnarray}
and consequently,
\begin{eqnarray}
&&\hskip-.8in p_x = - 2\overline{u u_x}.  \label{DTEQ-6}
\end{eqnarray}
Substituting (\ref{DTEQ-6}) into system  (\ref{TEQ-1})--(\ref{TEQ-3}) we
obtain the closed system
\begin{eqnarray}
&&\hskip-.8in \frac{\pp u}{\pp t} + u u_x  + w u_z - 2\overline{u u_x}  = 0,
\label{DTEQ-7}     \\ &&\hskip-.8in u_x + w_z =0,   \label{DDTEQ-3}
\end{eqnarray}
subject to the boundary conditions (\ref{PP-1})-(\ref{PP-3}).
In particular, by differentiating with respect to $x$ we have
\begin{eqnarray}
&&\hskip-.8in \frac{\pp u_x}{\pp t} + (u u_x)_x  + w_x u_z+wu_{xz} -
2\overline{(u u_x)_x}  = 0.  \label{DTEQ-8}
\end{eqnarray}
Let us consider the restriction of the evolution of equation (\ref{DTEQ-8}) on the line $x=0$. Since $u(x,z,t)$ is an odd-function, and $w(x,z,t)$ is an even-function, with respect to the variable $x$ we have:
\[
u(0,z,t)=0; w_x(0,z,t)=0\,.
\]
This together with (\ref{DTEQ-8}) imply
\begin{eqnarray}
&&\hskip-.8in u_{tx}(0,z,t) + ( u_x(0,z,t))^2 +w(0,z,t) u_{xz}(0,z,t) -
2\overline{(u_x(0,z,t))^2}  = 0.  \label{DTEQ-9}
\end{eqnarray}
Recalling that $w_z(0,z,t)=-u_x(0,z,t),$ and denoting by $W(z,t)=w(0,z,t),$ we
obtain
\begin{eqnarray}
&&\hskip-.8in W_{tz} - (W_z)^2 +W W_{zz} + \frac{2}{H}\int_0^H (W_z)^2 \,
dz  = 0,  \label{DTEQ-10}
\end{eqnarray}
with the boundary conditions:
\begin{eqnarray}
&&\hskip-.8in W (0,t)=  W (H,t)=0.  \label{DB}
\end{eqnarray}
This is a closed equation that we will investigate in section \ref{S-3}.

\section{Self-similar blowup of the reduced equation} \label{S-3}
In this section we will investigate the blowup of the nonlinear integro-differential boundary value problem (\ref{DTEQ-10})-(\ref{DB}). The analytical and numerical investigation of this type of equations have been the subject matter of the papers \cite{Okamoto,Okamoto-Zhu}; see also \cite{Ibrahim-Nakanishi-Titi} for additional new results. The presentation below follows \cite{Childress}  (see also \cite{Okamoto} section 4, and references therein), where  the same problem, arising in a different fluid dynamical context, has been investigated. For the sake of completeness we provide below the full details.

Noting that  equation (\ref{DTEQ-10})  is invariant for the scaling
\[
W(z,t) \mapsto \la W(z,\la t),
\]
therefore, we look for a self-similar solution in the form
\[
 W(z,t) = \frac{\fy(z)}{1-t}, \, \hbox{with}   \, \fy(0)=\fy(H)=0.
 \]
 This solution starts from the initial profile $W(z,0) =  \fy(z)$, and blows up at time $t=1$ at every $z\in (0,H)$, for which $\fy(z)\not=0$. Therefore, the idea  is to construct a nontrivial  initial profile $\fy(z)$. Using the above ansatz and substituting in   (\ref{DTEQ-10}) we obtain the following  reduced nonlinear nonlocal boundary value problem:
\EQ{\label{fy}
 \fy' - (\fy')^2 + \fy \fy'' + \frac{2}{H}\int_0^H(\fy'(z))^2dz = 0,
 \quad \fy(0)=\fy(H)=0\,,}
that we need to show that it has a  nontrivial solution  $\fy$.

Let $m>0$ be a given  free parameter, we consider instead of (\ref{fy}) the nonlinear boundary value problem
\EQ{\label{fym}
 \fy' - (\fy')^2 + \fy \fy'' + m^2 = 0,
 \quad \fy(0)=\fy(H)=0.}
We  observe that in case the boundary value problem  (\ref{fym})  has a solution, which we denote by $\fy_m$, then  $\fy_m$    is nontrivial because $m>0$. Moreover,  integrating equation (\ref{fym}) over $0<z<H$ yields
\[
 m^2= \frac{2}{H}\int_0^H(\fy'_m(z))^2dz.
 \]
 Consequently, $\fy_m$ is also a nontrivial  solution to the original nonlocal boundary value problem (\ref{fy}).  Therefore, we will focus now on showing that problem (\ref{fym}) has a nontrivial solution for every $m>0$ given.

First we look for a positive solution $\fy$ to the boundary value problem (\ref{fym}). Denoting by $\psi:=\fy'$, then any local smooth solution to the second order ordinary differential equation given in (\ref{fym}) can be parameterized by an integral  curve $(\fy(z),\psi(z))$,  in the phase portrait space $(\fy,\psi)$. Suppose that we can parameterize, say locally, the integral curve in terms of $\psi$ instead of $z$, i.e., that we can invert the relationship between   $\psi$  and $z$ and express $z$ as function $\psi$. Then equation (\ref{fym}) becomes
\[
 \frac{1}{\fy}\frac{d\fy}{d\psi}=\frac{\psi}{\psi^2-\psi-m^2}=
 \frac{\psi}{(\psi-\psi_+)(\psi-\psi_-)},
 \]
where $\psi_\pm$ are the singular points
\[
 \psi_+(m):=\sqrt{m^2+1/4}+1/2>0, \quad \psi_-(m):=-\sqrt{m^2+1/4}+1/2<0.
 \]
Integrating the above equation we obtain general integral curves in the phase portrait for $(\fy,\psi)$
\EQ{ \label{integral curve}
 \fy=C|\psi-\psi_+|^{\frac{\psi_+}{\psi_+-\psi_-}}|\psi-\psi_-|^{\frac{-\psi_-}{\psi_+-\psi_-}},}
with an integration constant $C$. Since we are interested in positive solutions $\fy$ we take $C>0$. In order for this curve to be a solution to the boundary value problem (\ref{fym}), and that it can be parameterized in terms of $z \in [0,H]$, one has to require that for each $C>0$ the above curve yields a solution satisfying
\EQ{ \label{sing bc}
 (\fy,\psi)(z=0)=(0,\psi_+), \quad (\fy,\psi)(z=H)=(0,\psi_-).}
 Injecting \eqref{integral curve} into  (\ref{fym}) yields
 \EQ{ \label{psi-prime}
\frac{d\psi}{dz}=  \frac{-1}{C} |\psi-\psi_+|^{\frac{-\psi_-}{\psi_+-\psi_-}}|\psi-\psi_-|^{\frac{\psi_+}{\psi+-\psi_-}}\,,
 }
 which shows that $\psi$ is a decreasing function and hence the inversion between  the variables  $z$ and $\psi$ is valid over the   interval $\psi_- \le \psi \le \psi_+$. In particular, for every $\psi \in [\psi_- , \psi_+]$  and by virtue of  (\ref{psi-prime}) and (\ref{sing bc}) we have
 \EQ{ \label{z}
 z(\psi)=\int_{\psi_+}^{\psi}\frac{dz}{d\psi}d\psi
 = \int_{\psi_+}^{\psi}\frac{\fy}{\psi+\psi^2-m^2}d\psi
= -C \int_{\psi_+}^{\psi}|\psi-\psi_+|^{\frac{\psi_-}{\psi_+-\psi_-}}|\psi-\psi_-|^{\frac{-\psi_+}{\psi+-\psi_-}}d\psi
\,.  }
Moreover, from the other boundary condition in (\ref{sing bc}) the length of interval is determined by
\EQ{\label{length}
 H&=z(\psi_-)= \int_{\psi_+}^{\psi_-}\frac{dz}{d\psi}d\psi
 = \int_{\psi_+}^{\psi_-}\frac{\fy}{\psi+\psi^2-m^2}d\psi
 \pr= C \int_{\psi_-}^{\psi_+}|\psi-\psi_+|^{\frac{\psi_-}{\psi_+-\psi_-}}|\psi-\psi_-|^{\frac{-\psi_+}{\psi+-\psi_-}}d\psi
 \pr= \frac{C}{\psi_+-\psi_-}B\left(\frac{\psi_+}{\psi_+-\psi_-},\frac{-\psi_-}{\psi_+-\psi_-}\right)
 \pr= \frac{C}{2\sqrt{m^2+1/4}}B\left(\frac 12+\frac{1}{4\sqrt{m^2+1/4}},\frac12 -\frac{1}{4\sqrt{m^2+1/4}}\right),}
where $B(\cdot,\cdot)$ is the Beta function.

Hence with the given $H>0$, for each $m>0$ there is a unique $C>0$, depending on $H$ and $m$,  such that (\ref{length}) holds.

From all the above we concluded that there is a  solution $\fy$  of (\ref{fym}) satisfying
\[
 \fy(0)=0=\fy(H), \quad \fy>0 \text{ on $0<z<H$}, \quad \hbox{and} \quad \frac{2}{H}\int_0^H(\fy'(z))^2dz=m^2.
 \]
 This in turn also shows that   $\fy$ is a nontrivial solution to the nonlinear nonlocal boundary value problem (\ref{fy}).

From (\ref{psi-prime}), (\ref{sing bc}) and the above discussion one also concludes that
\EQ{ \label{second vanish}
 \fy_{zz}(0)=0=\fy_{zz}(H).}
Furthermore, the symmetry of equation (\ref{fym})  implies that   $-\fy(H-z)$ is a negative solution to (\ref{fym}).

In conclusion, we have obtained a one-parameter family of  blowup solutions of (\ref{DTEQ-10})-(\ref{DB}), which blow up at every $z\in(0,H)$, as $t\to1^-$.

\begin{remark} If we choose  the constant $C$ in (\ref{length}) such that $z(\psi_-)= \frac{H}{2}$, then   condition \eqref{second vanish} holds at $z=\frac{H}{2}$ instead of $z=H$. This in turn will allow us to   glue the positive  solution defined on the interval $[0,\frac{H}{2}]$, and which is also satisfying \eqref{sing bc} $z=\frac{H}{2}$ instead of $z=H$, with its negative counterpart at $z=\frac{H}{2}$. Hence, this idea will allow us  construct  sign-changing blowup solutions.  In fact, this is the type of solutions that are constructed in \cite{Childress} and \cite{Okamoto}. Of course one can repeat the above and get profiles that change signs as many times as one wants.

\end{remark}

\section{Blowup of the smooth solutions to the primitive equations} \label{S-4}
In this section we will demonstrate our main result and show that for certain class of initial data  the corresponding smooth solutions to the  primitive equations (\ref{TEQ-1})-(\ref{TEQ-3}), subject to the boundary conditions (\ref{PP-1})-(\ref{PP-3}), blow up in finite time. First, let us prove the following proposition concerning the uniqueness of solutions to equation (\ref{DTEQ-10})-(\ref{DB}).

\begin{proposition}\label{unique-W}
Let $W_1, W_2$ be two solutions of (\ref{DTEQ-10})-(\ref{DB}) which are in $L^2 ((0,T);H^2)$, with $W_1(z,0) = W_2(z,0)$. Then $W_1(z,t) = W_2(z,t)$, for all $t\in [0,T)$. In particular, $W(z,t) =\frac{\fy(z)}{1-t}$ is the only solution of   (\ref{DTEQ-10})-(\ref{DB}) in the space $L^2 ((0,T);H^2)$, for all $T\in [0,1)$, with initial data $\fy(z)$, where   $\fy(z)$ is any   nontrivial solution of the boundary value problem (\ref{fy}) that was established in section \ref{S-3}.
\end{proposition}

\begin{proof} Let $V=W_1-W_2$, and $\overline{W}= \frac{1}{2}(W_1+W_2)$. Then (\ref{DTEQ-10})-(\ref{DB}) imply:
\begin{eqnarray}
&&\hskip-.8in V_{tz} - 2\overline{W}_z V_z+V \overline{W}_{zz} +\overline{W} V_{zz} + \frac{4}{H}\int_0^H \overline{W}_z V_z  \,
dz  = 0,  \label{DTEQ-10-Diff}
\end{eqnarray}
with the boundary conditions:
\begin{eqnarray}
&&\hskip-.8in V (0,t)=  V (H,t)=0.  \label{DB-Diff}
\end{eqnarray}
Multiplying (\ref{DTEQ-10-Diff}), integrating with respect to $z$ over $[0,H]$, integrating by parts and using the boundary conditions (\ref{DB}) and (\ref{DB-Diff}) gives:
\[
\frac{1}{2}\frac{d}{dt}\int_0^H V^2(z,t) dz = \frac{5}{2} \int_0^H \overline{W}_z (V_z)^2 dz - \int_0^H \overline{W}_{zz} V V_z dz \,.
\]
Thus, by H\"older inequality we obtain
\[
\frac{d}{dt} \|V_z(t)\|^2_{L^2(0,H)} \le   5\|\overline{W}_z\|_{L^\infty(0,H)}   \|V_z(t)\|^2_{L^2(0,H)}+  2 \|\overline{W}\|_{H^2(0,H)}  \|V(t)\|_{L^\infty(0,H)} \|V_z(t)\|_{L^2(0,H)}  \,.
\]
By virtue of the one-dimensional Sobolev imbedding theorem, and thanks to the boundary conditions (\ref{DB-Diff})  one can apply the Poincar\'e inequality, to obtain:
\[
\frac{d}{dt} \|V_z(t)\|^2_{L^2(0,H)} \le c \|\overline{W}\|_{H^2(0,H)} \|V_z(t)\|^2_{L^2(0,H)}\,.
\]
Applying Gronwall's inequality we then conclude that  $\|V_z(\cdot,t))\|_{L^2(0,H)} = 0$, for all $t \in [0,T)$. Again, thanks to  (\ref{DB-Diff}) we apply Poincar\'e inequality to conclude the uniqueness part of the statement of the proposition. The second part of the statement is an immediate corollary of the first part.

\end{proof}

\begin{theorem}
Let $\fy(z)$ be any  nontrivial solution of the boundary value problem (\ref{fy}), and let  $k$ be any positive integer. Suppose that  $(u(x,z,t),
w(x,z,t))$ is a smooth solution of  (\ref{TEQ-1})-(\ref{TEQ-3}), subject to the boundary conditions (\ref{PP-1})-(\ref{PP-3}), with initial value
\[
u_0(x,z) = \frac{-L}{2\pi k} \sin(\frac{2\pi k}{L} x) \fy'(z)  \quad \hbox{and} \quad w_0(x,z) =  \cos(\frac{2\pi k}{L} x) \fy(z) \,.
\]
Then this solution blows up in finite time.
\end{theorem}

\begin{proof}
Based on the derivation in section \ref{S-2} the function $w(0,z,t)$ satisfies (\ref{DTEQ-10})-(\ref{DB}), with initial value  $w(0,z,0)=\fy(z)$. Therefore, by   Proposition \ref{unique-W} one concludes that for as long as the solution $w(x,z,t)$ exists and is smooth must  satisfies  $w(0,z,t)=\frac{\fy(z)}{1-t}$. Consequently,  $w(x,z,t)$ must lose regularity at sometime in the interval $t \in [0,1]$.

\end{proof}

\section{Acknowledgments} The authors would like to thank the Banff International Research Station (BIRS), Canada, for the kind and warm hospitality where part of this work was completed.
The work of C. Cao is supported in part by the NSF grant no. DMS-1109022.  S. Ibrahim is partially supported by NSERC \#371637-2009 grant. E.S. Titi is  supported in part by the   NSF grants DMS-1009950, DMS-1109640,  and DMS-1109645. E.S. Titi  also acknowledges the support of the Alexander von Humboldt Stiftung/Foundation and the Minerva Stiftung/Foundation.

\end{document}